%% file: main.tex
\documentclass[11pt,a4paper]{article}
\pdfoutput=1
\usepackage{natbib}
\usepackage[margin=1in]{geometry}
\usepackage[ascii]{inputenc} % removed utf-8 to detect wrong accents
\usepackage[dvipsnames]{xcolor}
\usepackage[bookmarksnumbered,linktocpage,hypertexnames=false,colorlinks=true,linkcolor=NavyBlue,urlcolor=NavyBlue,citecolor=ForestGreen,anchorcolor=green,breaklinks=true,pagebackref=true,pdfusetitle]{hyperref}
\renewcommand*{\backref}[1]{}
\renewcommand*{\backrefalt}[4]{%
  \ifcase #1 %
    (Not cited.)%
  \or
    (Cited on pgs.~#2)%
  \else
    (Cited on pgs.~#2)%
  \fi}
\usepackage{bbm,braket,microtype,amsmath,amssymb,amsthm,amsfonts,dsfont,mathtools,colortbl,booktabs,algorithm,algpseudocode,graphicx,enumitem,xspace,float,mathdots,ellipsis,caption,subcaption,mleftright,multirow,adjustbox}
\usepackage[T1]{fontenc}
\usepackage{textcomp}
\usepackage[english]{babel}
\usepackage[capitalize,nameinlink]{cleveref}
\usepackage{tikz}
\newtheorem*{theorem*}{Theorem}
\newtheorem{theorem}{Theorem}[section]\crefname{theorem}{Theorem}{Theorems}
\newtheorem{lemma}[theorem]{Lemma}\crefname{lemma}{Lemma}{Lemmas}
\AddToHook{env/lemma/begin}{\crefalias{theorem}{lemma}}
\crefname{claim}{Claim}{Claims}
\AddToHook{env/claim/begin}{\crefalias{theorem}{claim}}
\newtheorem{proposition}[theorem]{Proposition}\crefname{proposition}{Proposition}{Propositions}
\AddToHook{env/proposition/begin}{\crefalias{theorem}{proposition}}
\crefname{observation}{Observation}{Observations}
\AddToHook{env/observation/begin}{\crefalias{theorem}{observation}}
\newtheorem{corollary}[theorem]{Corollary}\crefname{corollary}{Corollary}{Corollaries}
\AddToHook{env/corollary/begin}{\crefalias{theorem}{corollary}}
\crefname{conjecture}{Conjecture}{Conjecture}
\AddToHook{env/conjecture/begin}{\crefalias{theorem}{conjecture}}
\theoremstyle{definition}
\crefname{definition}{Definition}{Definitions}
\AddToHook{env/definition/begin}{\crefalias{theorem}{definition}}
\crefname{problem}{Problem}{Problems}
\AddToHook{env/problem/begin}{\crefalias{theorem}{problem}}
\newtheorem{remark}[theorem]{Remark}\crefname{remark}{Remark}{Remarks}
\AddToHook{env/remark/begin}{\crefalias{theorem}{remark}}
\crefname{example}{Example}{Examples}
\AddToHook{env/example/begin}{\crefalias{theorem}{example}}
\crefname{condition}{Condition}{Conditions}
\AddToHook{env/condition/begin}{\crefalias{theorem}{condition}}
\numberwithin{equation}{section}

\AddToHook{cmd/appendix/before}{\crefalias{section}{appendix}}

\DeclareMathOperator{\arcosh}{arcosh}
\DeclareMathOperator{\arsinh}{arsinh}

\DeclareMathOperator{\argmin}{argmin}

\usepackage{tikz}
\usetikzlibrary{positioning}
\usetikzlibrary{calc}

\newcommand{\email}[1]{\href{mailto:#1}{#1}}
\title{Negative curvature obstructs the existence of good barriers for interior-point methods}
\author{
    Christopher Criscitiello\thanks{Department of Statistics and Data Science, University of Pennsylvania, \email{crisciti@wharton.upenn.edu}}
    \and
    Harold Nieuwboer\thanks{Department of Mathematical Sciences, University of Copenhagen, \email{hani@math.ku.dk}}
    \and
    Michael Walter\thanks{Faculty of Physics and Faculty of Mathematics, Computer Science, and Statistics, Ludwig-Maximilians-Universit\"at M\"unchen, \email{michael.walter@lmu.de}}
}
\date{}

\allowdisplaybreaks[4]

\newcommand{\RR}{\mathbb{R}}

\newcommand{\ZZ}{\mathbb{Z}}

\newcommand{\HH}{\mathbb{H}}

\DeclarePairedDelimiter\norm{\lVert}{\rVert}

\DeclareMathOperator{\Hess}{Hess}

\newcommand{\analyticcenter}{\ensuremath{p_*}}

\newcommand{\bibnoop}[1]{}

\begin{document}

\maketitle

\begin{abstract}
    Interior-point methods (IPMs) are a cornerstone of Euclidean convex optimization, due to their strong theoretical guarantees and practical performance.  Motivated by scaling problems, recent work by Hirai and the last two authors (FOCS'23) extended IPMs to geodesically convex optimization on Hadamard manifolds.
    Crucially, the complexity of IPMs (both in Euclidean and Hadamard spaces) is governed by the \emph{barrier parameter} of the domain.
    Here we prove that already in hyperbolic space, several natural domains---including geodesic balls and triangles---have a barrier parameter that grows polynomially with the domain's diameter.
    By extension, the same holds for the positive-definite matrices and other symmetric Hadamard spaces.
    This growth implies a fundamental limitation: interior-point methods relying on barriers for a ball cannot efficiently solve challenging scaling problems, such as tensor scaling, where the domain's diameter can be exponentially large in the input size.
    Our results are partially inspired by, and complement, lower bounds on the condition number of geodesically convex functions established by Hamilton and Moitra (NeurIPS'21).
\end{abstract}


%=============================================================================
\section{Introduction}\label{sec:introduction}
%=============================================================================
Interior-point methods (IPMs) play a central role in modern convex optimization both due to their practical performance and their strong theoretical foundations.
For instance, the best algorithms for linear programming are based on the principles of IPMs (combined with powerful data structures~\citep{vandenbrandDeterministicLinearProgram2019,vandenbrandSolvingTallDense2020,chenMaximumFlowMinimumCost2025}).
\Citet{nesterov1994interior} introduced a general framework for IPMs based on~\emph{self-concordant barriers}, which yields efficient algorithms for a broad class of convex problems beyond linear programming.
A key feature of their theory is that the complexity of (short-step) IPMs is governed by the~\emph{barrier parameter} of the domain.  For many natural domains---such as Euclidean half-spaces, balls and ellipsoids---this parameter is small or even constant, enabling fast convergence rates.
Even when the feasible region has a complicated shape, it is often sufficient to construct a barrier for a simple set such as a ball containing the relevant sublevel sets of the objective.
Consequently, understanding barriers for balls is a basic prerequisite for applying interior-point methods to more complicated domains.

It is natural to ask whether interior-point methods (IPMs) can be extended beyond convex optimization on Euclidean spaces. Recently,~\citet{HNW} generalized the Nesterov-Nemirovski framework to~\emph{geodesically convex} (g-convex) optimization on Hadamard manifolds.
A key motivation for their work was to tackle \emph{noncommutative group optimization} or \emph{scaling problems} \citep{burgissernoncommutativeoptimization}, which have attracted growing attention due to a broad range of applications in computer science~\citep{Garg2020Operator,burgisser2018alternating,allenzhuoperatorscaling}, statistics~\citep{franksmoitra2020,franks2025nearoptimal}, and quantum information~\citep{tensorscaling2018,acuaviva2023minimal}.
Polynomial-time algorithms are known for certain scaling problems, such as matrix and operator scaling~\citep{Garg2020Operator,allenzhuoperatorscaling}. However, for others---most notably tensor scaling~\citep{tensorscaling2018}---no general efficient algorithms are known.
While \citet{HNW} successfully applied their g-convex IPMs to scaling problems, they did not obtain polynomial-time algorithms, and in particular achieved comparable iteration-complexity bounds to box-constrained Newton methods from~\cite{burgissernoncommutativeoptimization}.

The main bottleneck is the construction of suitable self-concordant barriers.
These must have well-controlled barrier parameters and admit efficient evaluation.
In Euclidean space, for example, a ball of radius~$R$ admits a self-concordant barrier with constant parameter---independent of both radius and dimension.
More generally, every bounded Euclidean convex domain admits a self-concordant barrier whose parameter is bounded only by the dimension~(\cite{nesterov1994interior,bubeckEntropicBarrierExponential2019,chewiEntropicBarrierNSelfConcordant2023}), and this is optimal in general~\cite[Prop.~2.3.6]{nesterov1994interior}.
In contrast, while \citet[Thms.~1.4 \& 1.6]{HNW} were able to construct a barrier for a ball of radius~$R$ in hyperbolic space and, more generally, in symmetric Hadmard spaces such as the positive-definite matrices, the parameter of their barrier scaled polynomially with~$R$.
This scaling is prohibitive for problems like tensor scaling, where the domain may have diameter exponential in the input size~\citep{franksreichenback2021}.
This leads to the central question of this paper:
\begin{center}
\emph{Do there exist self-concordant barriers for natural domains (e.g., triangles, balls) in the hyperbolic plane whose parameter scales sub-polynomially in the domain's diameter?}
\end{center}
We show that the answer is no: the negative curvature of hyperbolic space fundamentally limits the existence of ``good'' self-concordant barriers.
The same obstruction applies to any non-flat symmetric Hadamard space,%
\footnote{Any such space contains the hyperbolic plane as a totally geodesic submanifold~\citep[Sec~II.10]{bridsonmetric}. Restricting a barrier to a totally geodesic submanifold yields a barrier with the same parameter.}
such as the positive-definite matrices equipped with the affine-invariant metric and products thereof.

This obstruction implies that IPMs based on barrier functions for the ball are insufficient to solve tensor scaling---or related noncommutative optimization problems---in polynomial time.
In contrast, consider \emph{array scaling}---the commutative analogue of tensor scaling---which corresponds to a convex optimization problem in Euclidean space.
For array scaling, one can construct barrier functions with small parameter values (based on the barrier for the Euclidean ball, combined with other techniques), leading to algorithms with low iteration complexity~\citep[Sec.~3.2]{buergisser2020interiorpointmethodsunconstrainedgeometric}.
Indeed, array scaling is known to be solvable in polynomial time~\citep[Thm~3.1 \& Rem.~3.1]{NEMIROVSKI1999435}.

%-----------------------------------------------------------------------------
\subsection{Main results}\label{sec:mainresults}
%-----------------------------------------------------------------------------
We now summarize our main results in more detail.
To this end, if $F$ is a self-concordant barrier of an open\footnote{Barrier functions diverge as they approach the boundary of the domain, and are therefore naturally defined on open sets. Other literature may instead define barriers for closed domains as functions defined on their interior (with the same properties).} g-convex domain~$D$ with parameter~$\theta < \infty$, we shall say~$F$ is a \emph{$\theta$-barrier} for~$D$.
We refer to \cref{sec:preliminaries} for a review of the notions of the hyperbolic plane $\HH^2$, g-convexity, self-concordance, barrier, analytic center, and related concepts.
As explained earlier, while our obstructions are general, it suffices to state them for the hyperbolic plane.

The first domain that we consider is an equilateral triangle, that is, the convex hull of three equidistant points in the hyperbolic plane.

\begin{theorem}[Triangle] \label{thm:triangle}
    Let $D$ be an equilateral triangle in the hyperbolic plane with side lengths $L$.
    If $F$ is a $\theta$-barrier for $D$, then $\theta \geq \frac{L}{12} - \frac{1}{6}$.
\end{theorem}
\noindent We contrast this with the Euclidean case, where it is easy to write down a barrier with~$\theta = 3$ for arbitrary triangles in~$\RR^2$, namely by viewing a triangle as an intersection of three half-spaces and combining the associated standard logarithmic barriers.
It is possible to extend \cref{thm:triangle} to regular polygons in the hyperbolic plane with $n \geq 3$ sides, and the resulting lower bound is of the form~$\theta = \Omega(L / \!\log(n))$; see~\cref{sec:regular-n-gons} for details.

An (open) half-space in the hyperbolic plane is a set of the form~$\{p \in \HH^2 : \langle \log_q(p), v \rangle < 0\}$ for some fixed~$q \in \HH^2$ and~$v \in T_q\HH^2$.
These are geodesically convex, although this is not the case for general Hadamard manifolds.
Since a hyperbolic triangle is an intersection of three hyperbolic half-spaces, we obtain the following corollary:

\begin{corollary}[Half-space] \label{cor:halfspace}
    There are no $\theta$-barriers for half-spaces in the hyperbolic plane.
\end{corollary}
\noindent Barriers for half-spaces in Euclidean space serve as a fundamental building block for constructing more complex barriers.
In their absence, constructing barriers becomes substantially more challenging.
Indeed, for general Hadamard manifolds, half-spaces are not even g-convex.
It is an open question to characterize which g-convex sets admit a $\theta$-barrier.

Next, we consider convex domains which are not too far from being balls, in the sense that they contain a large inner ball, and are contained in a ball that is not substantially larger.%
\footnote{This is not the case for hyperbolic triangles, which can only contain balls of constant radius; this feature is used in an essential way in the proof of~\cref{thm:triangle}.}
Here we find a strong lower bound for barriers whose analytic center coincides with the balls' center:

\begin{theorem}\label{thm:ball_like}
    Let $D$ be an open g-convex subset of the hyperbolic plane.
    Assume there exist
    %\MW{removed ``open''}
    balls of radii $R \geq r$ centered on $\analyticcenter$ surrounding $D$:
    \[ B(\analyticcenter, r) \subseteq D \subseteq B(\analyticcenter, R). \]
    If $F$ is a $\theta$-barrier for $D$ with analytic center $\analyticcenter$, then $\theta \geq \frac{1}{2^6} \cdot \sqrt r \cdot \frac{r}{R} - \frac{1}{2}$.
\end{theorem}
\noindent This theorem readily extends to Hadamard manifolds of bounded curvature, see~\cref{rem:extensionballlike}.

In particular, for balls and horoballs we obtain the following corollaries:
\begin{corollary}[Ball] \label{cor:ball}
Let $D$ be a
%\MW{removed ``open''}
ball in the hyperbolic plane with radius~$R$.
If $F$ is a $\theta$-barrier for~$D$, then $\theta \geq \frac{1}{2^7} \sqrt{R} - \frac{1}{4}$.
\end{corollary}
\noindent This shows that even the most natural domain in hyperbolic space does not admit good self-concordant barriers.
It is an open question whether our lower bound $\theta \geq \Omega(\sqrt{R})$ for balls is tight; the best known upper bound is $\theta \leq O(R^2)$~\citep[Thms.~1.4 \& 1.6]{HNW}.
\begin{corollary}[Horoball] \label{cor:horoball}
    There are no $\theta$-barriers for horoballs in the hyperbolic plane.
\end{corollary}
\noindent
In Euclidean space, horoballs are simply half-spaces.
In general, horoballs are defined as the sublevel sets of Busemann functions.
The latter also arise as certain special Kempf--Ness functions, % (associated with a highest weight vector of an irreducible representation),
see~\citep[Lem.~3.12]{hirai1} and \citep{franks2020minimal}.
Notably, Kempf--Ness functions are the objectives to be minimized in scaling problems \citep{burgissernoncommutativeoptimization}.
\Cref{cor:horoball} is a special instance of the following more general result.

\begin{corollary} \label{cor:horoconvexset}
    Let $D \subsetneq \HH^2$ be an open horospherically convex\footnote{For instance, $D$ may be an intersection of a possibly infinite number of open horoballs and balls (of potentially differing radii).  See \cref{sec:preliminariesSec2} for a general definition.} set in the hyperbolic plane.
    Assume $D$ contains a ball of radius $r$.
    If $F$ is a $\theta$-barrier for $D$, then $\theta \geq \frac{1}{2^9} \sqrt{r} - \frac{1}{8}$.
\end{corollary}
\noindent It is an open question whether~\cref{cor:horoconvexset} can be extended to geodesically convex sets.

We emphasize that all our results hold even if we only insist that $F$ is a self-concordant barrier \emph{along geodesics}, a weaker condition than full self-concordance (see~\citep[Sec.~3.2]{HNW}).
This is evident from the proofs, which we present after some brief preliminaries.

%=============================================================================
\section{Preliminaries and notation}\label{sec:preliminaries}
%=============================================================================
We briefly review the required background on Hadamard manifolds, geodesic and horospheric convexity, and self-concordant barriers, and introduce our notation.

%-----------------------------------------------------------------------------
\subsection{Hadamard manifolds and the hyperbolic plane}
%-----------------------------------------------------------------------------
Let $M$ be a Hadamard manifold---a complete simply connected smooth Riemannian manifold with nonpositive curvature, see~\citep{bridsonmetric} or~\citep[Ch.~12]{lee2018riemannian}.
The manifold~$M$ has tangent bundle~$T M$ and tangent spaces~$T_p M$.
The Riemannian metric on $M$ is denoted $\langle {\cdot}, {\cdot} \rangle$.
The Riemannian metric gives $M$ a notion of distance $d(\cdot, \cdot)$ as well as its geodesics (which generalize the straight lines in Euclidean space).
For $p, q \in M$, $\tau_{p \rightarrow q} \colon T_p M \rightarrow T_q M$ denotes parallel transport along the unique geodesic segment connecting $p$ and $q$.
The boundary of a set~$D \subseteq M$ is denoted $\partial D$.
The exponential map of~$M$ at~$p \in M$ is denoted by $\exp_p \colon T_p M \to M$, and its inverse map by $\log_p \colon M \to T_p M$.
As $M$ is Hadamard, the exponential map and its inverse are global diffeomorphisms.
An open \emph{half-space} is a set of the form
\begin{align}\label{eq:halfspace}
\{p \in M : \langle \log_q(p), v \rangle < 0\} = \{ \exp_q(w) : w \in T_qM, \langle w, v \rangle < 0 \}
\end{align}
for some fixed $q \in M$ and $v \in T_qM$.
The open \emph{ball} of radius $r$ and center $c$ in $M$ is denoted
\[ B(c, r) = \{p \in M : d(c, p) < r\}. \]
We denote the closed ball by $\bar{B}(r,c)$.
To every geodesic ray $\gamma \colon [0, \infty) \to M$, we can associate an open \emph{horoball}, which is a certain open set whose boundary contains $\gamma(0)$.
Formally, given a unit-speed geodesic ray $\gamma$, we define the Busemann function $b_\gamma(p) = \lim_{t\to\infty} d(\gamma(t), p) - t$.
Then the associated horoball is the sublevel set~$b_\gamma^{-1}((-\infty, 0))$~\citep[Sec.~II.8]{bridsonmetric}.
In Euclidean space, Busemann functions are affine functions and horoballs are half-spaces.

The hyperbolic plane $\HH^2$ is the two-dimensional Hadamard manifold with constant negative (sectional) curvature $-1$, see~\citep{ratcliffeFoundationsHyperbolicManifolds2019}.
Across every geodesic line $\gamma$ in $\HH^2$ there is a reflection isometry which fixes $\gamma$, see~\citep[Ch 5]{ratcliffeFoundationsHyperbolicManifolds2019} or~\citep[Sec.~I.6]{bridsonmetric} for the isometries of hyperbolic space.
In the Poincar\'e disk model for~$\HH^2$, balls are simply Euclidean balls in the interior of the unit disk, while horoballs are Euclidean balls tangent to its boundary, that is, to the unit circle (see \cref{fig:test2}).
We will need the hyperbolic law of cosines: in a hyperbolic triangle with sides of length~$a, b$ and angle~$\theta$ between them, the length~$c$ of the third side is given by
\begin{align}\label{eq:hyperpythag}
\cosh(c) = \cosh(a) \cosh(b) - \sinh(a) \sinh(b) \cos(\theta).
\end{align}
Any symmetric Hadamard space that is not flat, that is, not a Euclidean space, contains the hyperbolic plane as a totally geodesic submanifold (up to an overall rescaling of the metric).

%-----------------------------------------------------------------------------
\subsection{Geodesic convexity}\label{sec:preliminariesSec2}
%-----------------------------------------------------------------------------
A subset $D \subseteq M$ is called \emph{geodesically convex (g-convex)} if for all $p, q \in M$, the geodesic segment connecting $p$ and $q$ lies in $D$.
We call a closed subset $D$ \emph{horospherically convex} if it equals a (possibly uncountable) intersection of closed horoballs.
An open horospherically convex set is the interior of a closed horospherically convex set.
All horospherically convex sets are g-convex, but the converse is not true.
Balls and horoballs are horospherically convex in any Hadamard manifold.
Half-spaces~\eqref{eq:halfspace} in hyperbolic space are g-convex, but not horospherically convex.
However, in general Hadamard manifolds, half-spaces are not even g-convex.
For references on horospherical convexity, see~\citep{goodwin2024subgradient,criscitiello2025horosphericallyconvexoptimizationhadamard}.

For $F \colon M \to \mathbb{R}$, we denote its Riemannian gradient and Hessian at $p$ by $\nabla F(p) \in T_p M$ and $(\nabla^2 F)_p \colon T_p M \times T_p M \to \mathbb{R}$, respectively (see \citep{AMS08,boumal2020intromanifolds} for definitions).
The Riemannian Hessian evaluated along a direction $v \in T_p M$ is denoted $(\nabla^2 F)_p(v,v)$.
Similarly, we denote the Riemannian third derivative at $p$ by $(\nabla^3 F)_p \colon T_p M \times T_p M \times T_p M \to \mathbb{R}$.
Unlike the Euclidean third derivative, $(\nabla^3 F)_p(\cdot, \cdot, \cdot)$ is not symmetric in its inputs, see~\citep[Sec 3.1]{HNW} or~\citep[Rem 3.2]{criscitiello2020accelerated}.

A function $F \colon D \to \mathbb{R}$ is \emph{g-convex} if for every unit-speed geodesic segment~$\gamma$ contained in $D$, the composition $t \mapsto f(\gamma(t))$ is convex.
If $F$ is differentiable, this is equivalent to $F(q) \geq F(p) + \langle {\nabla F(p)}, {\log_p(q)}\rangle$ for all $p, q \in M$, and if $F$ is twice differentiable, to $(\nabla^2 F)_p \succeq 0$ for all $p \in M$.
Similary, $F$ is \emph{$\mu$-strongly convex} ($\mu>0$) if $t \mapsto f(\gamma(t))$ is $\mu$-strongly convex for every~$\gamma$ as above.
If $F$ is differentiable, $F$ is $\mu$-strongly g-convex in $M$ if and only if
\begin{align}\label{eq:strongconvexity}
  F(q) \geq F(p) + \langle {\nabla F(p)}, {\log_p(q)}\rangle + \frac{\mu}{2} d(p, q)^2 \quad \quad \forall p, q \in M.
\end{align}
If $F$ is twice differentiable, $F$ is $\mu$-strongly g-convex in $M$ if and only if $(\nabla^2 F)_p \succeq \mu I$ for all~$p \in M$.
Distance functions $x \mapsto d(x,p)$ are g-convex, and squared distances $x \mapsto d(x,p)^2$ are $2$-strongly g-convex.
For references on g-convex optimization see \citep{udriste1994convex},~\citep{bacak2014hadamard} or~\citep[Ch.~11]{boumal2020intromanifolds}.

%-----------------------------------------------------------------------------
\subsection{Self-concordant barriers}
%-----------------------------------------------------------------------------
Let us define self-concordant barriers on Hadamard manifolds $M$, as introduced by~\citep[Secs.~3 \& 4]{HNW}.
Let $D$ be an open g-convex subset of $M$.
A function $F \colon D \to \mathbb{R}$ is \emph{strongly $1$-self-concordant} if for all $p \in D$ and $u,v,w \in T_p M$
\begin{align}\label{eq:selfconcordance}
|(\nabla^3 F)_p(u, v, w)| \leq 2 \sqrt{(\nabla^2 F)_p(u,u)}\sqrt{(\nabla^2 F)_p(v,v)}\sqrt{(\nabla^2 F)_p(w,w)},
\end{align}
and its epigraph $\{(p, t) : p \in D, t\geq F(p) \}$ is a closed subset of $M \times \mathbb{R}$.
We say $F$ is strongly $1$-self-condordant \emph{along geodesics} if~\eqref{eq:selfconcordance} holds when $u=v=w$ (this is a weaker condition).

A function $F \colon D \to \mathbb{R}$ is a \emph{$\theta$-barrier (along geodesics)}, $\theta \geq 0$, if $F$ is strongly $1$-self-concordant (along geodesics), has $(\nabla^2 F)_p \succ 0$ for all $p \in D$,%
\footnote{The assumption $(\nabla^2 F)_p \succ 0$ is not strictly necessary, due to~\cite[Cor 3.9]{HNW}.} and
\[ \max_{0 \neq v \in T_p M} \frac{\langle \nabla F(p), v\rangle}{\sqrt{(\nabla^2 F)_p(v,v)}} \leq \sqrt{\theta} \quad \quad \forall p \in D. \]
The quantity~$\theta$ is called the \emph{barrier parameter} of $F$, and plays a crucial role in the complexity of \emph{barrier-generated path-following methods}.
If a $\theta$-barrier $F \colon D \to \mathbb{R}$ attains its minimum on $D$ (this is the case, e.g., if $D$ is bounded~\citep[Cor.~3.19]{HNW}), then the minimizer is unique and is called the \emph{analytic center}.

We list known properties of barriers which we use in our proofs.
Let $F$ be a $\theta$-barrier along geodesics for $D \subseteq M$.
\begin{enumerate}[label=(\textbf{P\arabic*})]
\item \label{P1} \cite[Thm.~3.4]{HNW}
If $p \in D$ and $u \in T_p M$ with $(\nabla^2 F)_p(u,u) < 1$, then~$q = \exp_p(u) \in D$, and
\[ (\nabla^2 F)_p(u,u) \geq \Big(1-\sqrt{(\nabla^2 F)_p(u, u)}\Big)^2 (\nabla^2 F)_q(\tau_{p \to q} u, \tau_{p \to q} u). \]

\item \label{P2} \citep[Cor.~3.8]{HNW}
If $p \in D$, $u \in T_p M$ with $\exp_x(u) \not \in D$, then $(\nabla^2 F)_p(u,u) \geq 1$.

\item \label{P3} \citep[Prop.~4.7]{HNW} If $F$ is bounded from below, with analytic center $\analyticcenter \in D$, then
\[ (\nabla^2 F)_{\analyticcenter}(\log_{\analyticcenter}(p),\log_{\analyticcenter}(p)) \leq (2\theta+1)^2 \quad \quad \forall p \in D. \]

\item \label{P4} \citep[Lem.~4.2]{HNW} If $F_i \colon D_i \to \mathbb{R}$ is a $\theta_i$-barrier (along geodesics), for $i=1,2$, then $F_1+F_2$ is a $(\theta_1 + \theta_2)$-barrier (along geodesics) for $D_1\cap D_2$.
\end{enumerate}
\ref{P1} and~\ref{P2} follow from $1$-self-concordance along geodesics.  \ref{P3} follows from the $\theta$-barrier property along geodesics. \ref{P4} follows from the sum of $1$-self-concordant (along geodesics) functions being $1$-self-concordant (along geodesics), as well as a simple estimate on the barrier parameter of a sum of barriers.

%=============================================================================
\section{Triangles and half-spaces}\label{sec:triangles}
%=============================================================================
In this section we provide a lower bound on the parameter of a self-concordant barrier for an equilateral triangle in the hyperbolic plane.

\begin{proof}[Proof of~\cref{thm:triangle} (Equilateral triangles)]
    Let $D$ be the interior of an equilateral triangle with side lengths $L$, and let $F \colon D \to \mathbb{R}$ be a $\theta$-barrier.
    Because $D$ is bounded, we know that $F$ has an analytic center.

    First we reduce to the situation that the analytic center coincides with the centroid of the triangle.
    To this end, let $G$ be the subgroup of the symmetry group of~$D$ which corresponds to rotations by $2\pi/3$ about the centroid (it is isomorphic to~$\ZZ/3\ZZ$).
    Then $\tilde F(p) = \sum_{g \in G} F(g \cdot p)$ is a $\tilde \theta$-barrier
    for $\tilde\theta = 3 \theta$, by property~\ref{P4}.
    Moreover, the gradient of $\tilde F$ at the centroid equals the sum of three rotated copies of the gradient of $F$ at the centroid, and is therefore equal to zero.
    Therefore, by g-convexity of $\tilde F$, the centroid must be the analytic center~$\analyticcenter \in D$ of~$\tilde F$.

    Let the vertices of the triangle be $a,b,c$.
    Consider the geodesic segment which passes through~$\analyticcenter$ and~$a$, and intersects the side opposite to~$a$ at the point~$q$, see \cref{fig:test1}.
    Using the hyperbolic law of cosines~\eqref{eq:hyperpythag} twice,%
\footnote{First on the isosceles triangle $b a \analyticcenter$ to determine $d(\analyticcenter, a) = d(\analyticcenter, b)$, using that the angle at~$\analyticcenter$ equals~$2\pi/3$, and then on the right triangle $b q \analyticcenter$ to determine $d(\analyticcenter, q)$.} we find
    \[ \cosh(d(\analyticcenter, q))^2 = \frac{\cosh(L) - \cos(\frac{2\pi}{3})}{(1 - \cos(\frac{2\pi}{3})) \cosh(L/2)^2} = \frac{4 \cosh (L)+2}{3 \cosh (L)+3} < \frac{4}{3}. \]
    Therefore, $d(\analyticcenter, q) < \arcosh(2/\sqrt{3}) < 1$.

    As $bq\analyticcenter$ is a right triangle, we also have
    \[ \cosh(d(\analyticcenter, b)) = \cosh(d(b,q)) \cosh(d(\analyticcenter,q)) = \cosh(L/2) \cosh(d(\analyticcenter,q)) > \cosh(L/2). \]
    Since~$d(\analyticcenter,a) = d(\analyticcenter, b)$, we conclude that~$d(\analyticcenter,a) > \frac{L}{2}$.

    Let $u = \log_{\analyticcenter}(q)$.  Since $q = \exp_{\analyticcenter}(u) \not \in D$, property~\ref{P2} implies $(\nabla^2 \tilde F)_{\analyticcenter}(u, u) \geq 1$.
    On the other hand, $\exp_{\analyticcenter}(- \frac{L}{2} \cdot \frac{u}{\|u\|}) \in D$ because $d(\analyticcenter, a) > \frac{L}{2}$.
    Therefore property~\ref{P3} implies
    \[ \frac{L^2}{4 \|u\|^2}(\nabla^2 \tilde F)_{\analyticcenter}(u, u) = (\nabla^2 \tilde F)_{\analyticcenter}\Big(\frac{L}{2} \cdot \frac{u}{\|u\|}, \frac{L}{2} \cdot \frac{u}{\|u\|}\Big) \leq (2 \tilde \theta + 1)^2. \]
    Using $\|u\| = d(\analyticcenter, q) \leq 1$ and $(\nabla^2 \tilde F)_{\analyticcenter}(u, u) \geq 1$, we conclude $6 \theta + 1 = 2 \tilde \theta + 1 \geq L/2$.
\end{proof}

The proof of~\cref{thm:triangle} is quite general.
It applies to any domain $D$ with analytic center~$\analyticcenter$ for which there is a geodesic segment~$\gamma$ through~$\analyticcenter$, satisfying:
(1)~in one direction along~$\gamma$ all points in~$D$ are~$O(1)$-close to~$\analyticcenter$, and~(2) in the other direction there are points in~$D$ which are~$\Omega(L)$-far away from~$\analyticcenter$.
Now we turn our attention to hyperbolic half-spaces.

\begin{proof}[Proof of \cref{cor:halfspace} (Half-spaces)]
For contradiction, assume there is a $\theta$-barrier for some hyperbolic half-space~$D$.
Then, by symmetry, every half-space in $\HH^2$ has a $\theta$-barrier.
Take three half-spaces, with corresponding $\theta$-barriers $F_1, F_2, F_3$, and arrange them such that their intersection forms the interior of an equiliteral triangle with side-lengths~$L$ satisfying~$\frac{L}{12} - \frac{1}{6} > 3 \theta$.
Then $F_1 + F_2 + F_3$ is a $3\theta$-barrier for the triangle, by property~\ref{P4}.
By \cref{thm:triangle}, we must have~$\frac L {12} - \frac 1 6 < 3 \theta$.
This is the desired contradiction.
\end{proof}

\begin{remark}
Let $D \subseteq \HH^2$ be an open half-space.
One might expect that $F \colon p \mapsto -\log(d(p, \partial D))$ is a barrier for $D$.
However, $F$ is not even g-convex!
In fact: the following proposition shows that any g-convex function on~$D$ that only depends on the distance~$d(p, \partial D)$ to the boundary will never increase towards the boundary!
That is, even the most basic requirements of self-concordant barriers, namely convexity and blowing up near the boundary, are incompatible.
\end{remark}

\begin{proposition}
    \label{thm:halfspace}
    Consider an open half-space~$D$ in the hyperbolic plane.
    Let~$f\colon \RR_{>0} \to \RR$ be any smooth function.
    If~$F(x) = f(d(x, \partial D))$ is convex, then~$f$ is non-decreasing: if~$0 < s \leq r$, then~$f(s) \leq f(r)$.
    % In other words, $F$ ``decreases'' towards the boundary!
\end{proposition}
\begin{proof}
Fix a point $p \in D$.
Let $q$ be the metric projection of $p$ onto $\partial D$, i.e., $q = \argmin_{q' \in \partial D} d(p,q')$.
Let $c \colon [0,1] \to \HH^2$ be the geodesic segment from $p$ to $q$ (so $c$ is orthogonal to $\partial D$).
Consider the geodesic $\gamma$ with $\gamma(0) = p$ and orthogonal to $c$.
By reflection symmetry of $\HH^2$ across $c$, $t \mapsto F(\gamma(t)) = f(d(\gamma(t), \partial D))$ is an even convex function.
In particular, it attains its minimum at~$t=0$ and is non-decreasing on $[0, \infty)$.
On the other hand, hyperbolic trigonometry shows that $d(\gamma(t), \partial D)$ is strictly increasing for $t \geq 0$, with image~$[d(p, \partial D), \infty)$.
We conclude $f$ is non-decreasing on this interval, hence on all of~$\RR_{>0}$ (since $p$ was arbitrary).
\end{proof}

\begin{figure}
\centering
\begin{minipage}{.47\textwidth}
  \centering
  \includegraphics[
    width=\linewidth,
    trim=9cm 3cm 9cm 2.5cm,
    clip
  ]{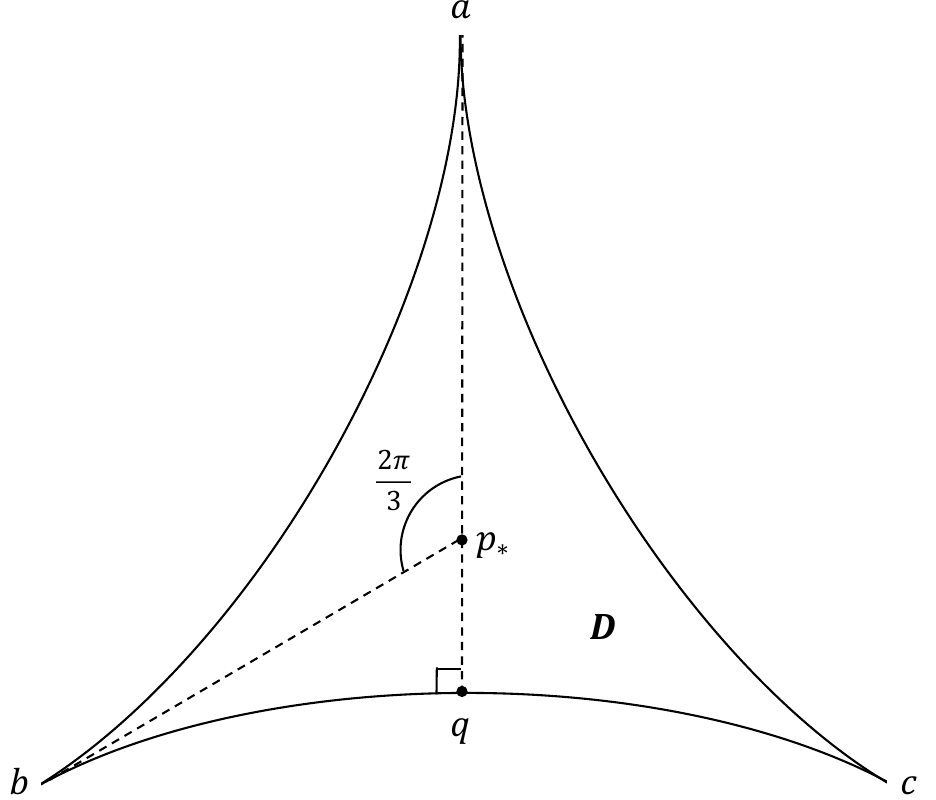}
  \captionof{figure}{A hyperbolic equilateral triangle, see \cref{sec:triangles}.}
  \label{fig:test1}
\end{minipage}%
\hfill
\begin{minipage}{.47\textwidth}
  \centering
  \includegraphics[
    width=0.9\linewidth,
    trim=9cm 2cm 9cm 2cm,
    clip
  ]{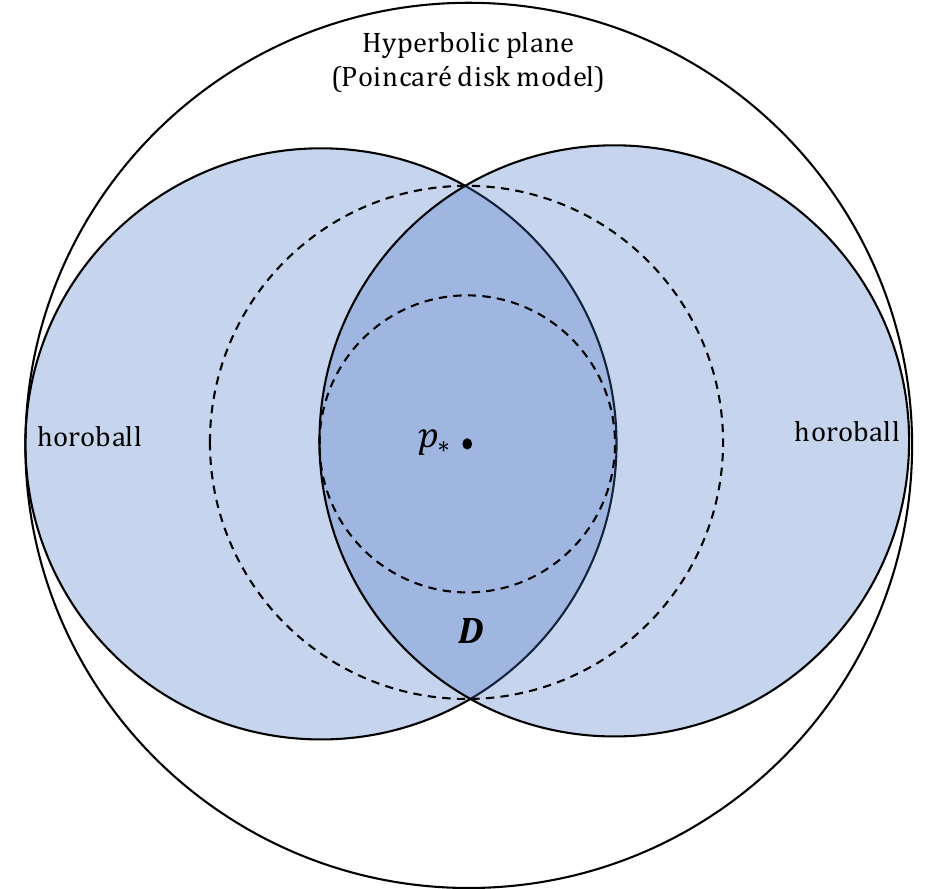}
  \captionof{figure}{Intersection of two horoballs in the hyperbolic plane, see \cref{sec:horoball}.}
  \label{fig:test2}
\end{minipage}
\end{figure}

%=============================================================================
\section{Balls and horoballs}\label{sec:balls}
%=============================================================================
In this section we provide lower bounds on the parameter of self-concordant barriers for ball-like domains~$D \subseteq \HH^2$ in the hyperbolic plane.
The strategy is conceptually similar to the argument for the triangle: using property~\ref{P2}, we first show that there exists a point $p \in D$ and unit direction $v \in T_p \HH^2$ where the Hessian $(\nabla^2 F)_p(v,v)$ is sufficiently large.
Moreover, $p$ is chosen such that it is the minimizer of~$F$ restricted to the geodesic in the direction of~$v$, i.e., the analytic center.
We then appeal to property~\ref{P3} to lower bound the parameter $\theta$.
Part of the proof is inspired by Theorem~6 in~\citep{hamilton2021nogo}, which gives a lower bound on the condition number of a g-convex function.

\begin{proof}[Proof of \cref{thm:ball_like}]
If $r < 2$, then the bound $\theta \geq \frac{1}{2^6} \cdot \sqrt r \cdot \frac{r}{R} - \frac{1}{2}$ holds trivially, as the right-hand side is negative.
Therefore, we assume $r \geq 2$ in the remainder of the proof.
% $F$ is a $\theta$-barrier for $D \subseteq \HH^2$ with analytic center $\analyticcenter$ (which exists since $D$ is bounded).
Without loss of generality, we also assume $F(\analyticcenter) = 0$.

We first establish that~$(\nabla^2 F)_p \succeq \frac{1}{R^2} I$ for all~$p \in D$.\footnote{From the triangle inequality, one can deduce~$D \subseteq B(p,2R)$, which combined with property~\ref{P2} leads to the weaker bound~$(\nabla^2 F)_p \succeq \frac{1}{(2R)^2} I$.}
Let~$p \in D$ and~$u \in T_p \HH^2$ such that~$(\nabla^2 F)_p(u,u) = 1$.
Let~$v = \log_p(\analyticcenter)$. Then either~$u$ or $-u$ has non-positive inner product; without loss of generality, assume that it is $u$.
Property~\ref{P2} implies that $q = \exp_p(u) \in \overline{D}$.
Therefore~$d(\analyticcenter,q) \leq R$, and
\begin{align*}
  \cosh(R) \geq \cosh(d(\analyticcenter, q)) & = \cosh(d(p,q)) \cosh(d(p, \analyticcenter)) - \sinh(d(p,q)) \sinh(d(p,\analyticcenter)) \frac{\langle u, v \rangle}{\norm{u} \norm{v}} \\
                                             & \geq \cosh(d(p,q)) \cosh(d(p, \analyticcenter))
\end{align*}
by \eqref{eq:hyperpythag} and~$\langle u,v \rangle \leq 0$.
As~$\cosh(d(p,\analyticcenter)) \geq 1$, we conclude that~$d(p,q) \leq R$, hence~$\norm{u} \leq R$. Thus we have shown that $(\nabla^2 F)_{p}(u,u) = 1$ implies~$\norm{u}^2 \leq R^2$, which is equivalent to the bound~$(\nabla^2 F)_p \succeq \frac{1}{R^2} I$.

This establishes that $F$ is $\frac1{R^2}$-strongly g-convex.
Therefore, equation~\eqref{eq:strongconvexity} gives
\begin{equation}\label{eq_quadratic_growth}
F(p) \geq \frac{1}{2 R^2} d(p, \analyticcenter)^2 \quad \quad \forall p \in D.
\end{equation}
Let $p$ be a minimizer of $F$ over the hyperbolic circle $\partial B(\analyticcenter, \frac{r-1}{2})$.
Then for any $q \in \partial B(\analyticcenter, \frac{r}{2})$,
\begin{equation} \label{eq_lower_bound_q}
\frac{r-1}{r}F(q) = \frac{1}{r}F(\analyticcenter) + \frac{r-1}{r}F(q) \geq F(p)
\end{equation}
by convexity of $F$ (along the geodesic segment from~$\analyticcenter$ to~$q$) and minimality of $p$.

Let $\gamma$ be the unit-speed geodesic passing through $p$, $\gamma(0) = p$, and tangent to $\partial B(\analyticcenter, \frac{r-1}{2})$.
Let $q_{\pm}$ be the two intersection points of $\gamma$ with the circle $\partial B(\analyticcenter, \frac{r}{2})$: see~\cref{figure_ball}.
Defining
\[ \beta = \max_{|s| \leq d(p, q_+)} (\nabla^2 F)_{\gamma(s)}(\gamma'(s), \gamma'(s)), \]
we then have $F(q_\pm) \leq F(p) + \langle \nabla F(p), \log_p(q_\pm)\rangle + \frac{\beta}{2} d(p, q_\pm)^2$ by the Taylor expansion.
Summing these two inequalities and using~$\log_p(q_-) = -\log_p(q_+)$ (or $\langle \nabla F(p), \log_p(q_\pm) \rangle = 0$ by minimality) and $d(p,q_+) = d(p,q_-)$
yields
\begin{equation}\label{eq_pm}
\frac{F(q_+) + F(q_-)}{2} \leq F(p) + \frac{\beta}{2} d(p, q_+)^2.
\end{equation}
Combining this with inequality~\eqref{eq_lower_bound_q} and rearranging yields
\[ \frac{1}{r-1} F(p) \leq \frac{\beta}{2} d(p, q_+)^2. \]
Plugging in inequality~\eqref{eq_quadratic_growth} yields
\begin{equation}\label{eq_lower_bound_beta}
\frac{r-1}{8 R^2} = \frac{1}{r-1} \cdot \frac{1}{2 R^2} d(p, \analyticcenter)^2 \leq \frac{\beta}{2} d(p, q_+)^2 \leq 2 \beta.
\end{equation}
For the final inequality in~\eqref{eq_lower_bound_beta}, it is easy to show $d(p, q_+) \leq 2$.
Indeed, the hyperbolic law of cosines~\eqref{eq:hyperpythag} for the triangle~$\analyticcenter{}pq_+$ implies $d(p, q_+) = \arcosh(\frac{\cosh(r/2)}{\cosh((r-1)/2)}) \leq \arcosh(\sqrt{e}) \leq 2.$\footnote{Alternatively see~\citep[Sec 5]{hamilton2021nogo} or~\citep[App I]{bumpfctpaper}.}
Crucially, this bound on $d(p, q_+)$ is due to the negative curvature of the hyperbolic space~$\HH^2$; in Euclidean space~$\mathbb{R}^2$, we would instead have $d(p, q_+) \sim \sqrt{r}$.

\begin{figure}
  \centering
  \includegraphics[
    width=\linewidth,
    trim=1.5cm 1.5cm 3cm 1.5cm,
    clip
  ]{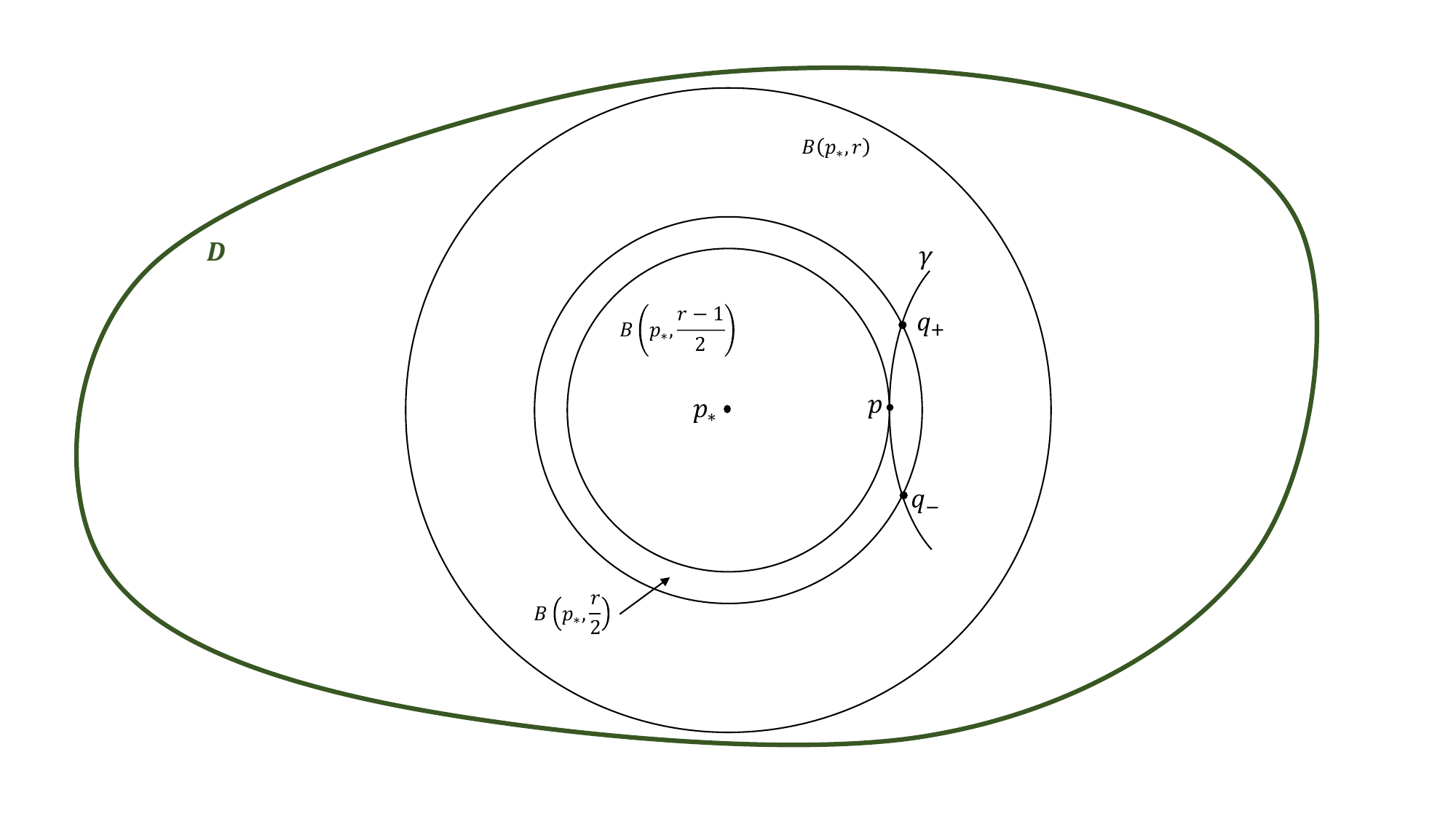}
  \caption{Geometric construction in $\HH^2$ for the proof of~\cref{thm:ball_like}.  See \cref{sec:balls}.}
  \label{figure_ball}
\end{figure}

By the definition of $\beta$, we conclude that there exists $s \in [-d(p, q_+), d(p, q_+)]$ with
\[ (\nabla^2 F)_{\gamma(s)}(\gamma'(s), \gamma'(s)) \geq \frac{r-1}{16 R^2}. \]
We can transfer this lower bound on $(\nabla^2 F)_{\gamma(s)}$ to a lower bound on $(\nabla^2 F)_{\gamma(0)}$.
We have two cases:
\begin{itemize}
\item If $(\nabla^2 F)_p(s \gamma'(0), s \gamma'(0)) < \frac{1}{4}$, then Property~\ref{P1} implies
\begin{equation}\label{eq_case2}
\begin{split}
  s^2 (\nabla^2 F)_p(\gamma'(0), \gamma'(0)) & = (\nabla^2 F)_p(s\gamma'(0), s\gamma'(0)) \\
&\geq \bigg(1-\sqrt{(\nabla^2 F)_p(s \gamma'(0), s \gamma'(0))}\bigg)^2 (\nabla^2 F)_{\gamma(s)}(s \gamma'(s), s \gamma'(s)) \\
&\geq \frac{1}{4} \cdot \frac{r-1}{16 R^2} \cdot s^2
\end{split}
\end{equation}
where we have used that $\exp_p(s \gamma'(0)) = \gamma(s)$ and $\tau \gamma'(0) = \gamma'(s)$ since $\gamma$ is a geodesic.
Therefore~$(\nabla^2 F)_p(\gamma'(0), \gamma'(0)) \geq (r-1) / (4 \cdot 16 R^2)$.

\item If $(\nabla^2 F)_p(s \gamma'(0), s \gamma'(0)) \geq \frac{1}{4}$, then
\begin{equation}\label{eq_case1}
(\nabla^2 F)_p(\gamma'(0), \gamma'(0)) \geq \frac{1}{4 s^2} \geq \frac{1}{16}
\end{equation}
since $|s|\leq d(p, q_+) \leq 2$.
\end{itemize}

Let $F|_\gamma$ denote the restriction of $F$ to the geodesic $\gamma$, which recall passes through~$p$ and~$q_\pm$.
Then $F|_\gamma$ is a $\theta$-barrier for the segment of $\gamma$ contained in $D$.
By minimality of $p$ on~$\partial B(\analyticcenter, \frac{r-1}{2})$, we know $\nabla (F|_\gamma)(p) = 0$ (i.e., $\nabla F (p)$ is orthogonal to $\gamma$), and so $p$ is the analytic center of $F|_\gamma$.
Therefore Property~\ref{P3} tells us that, for all $t \in \RR$ such that $\gamma(t) \in D$,
\[ (2 \theta + 1)^2
\geq (\nabla^2 (F|_\gamma))_p(t\gamma'(0), t\gamma'(0))
= (\nabla^2 F)_p(\gamma'(0), \gamma'(0)) d(p, \gamma(t))^2. \]
In particular, choosing $t$ so that $\gamma(t) \in \partial B(\analyticcenter, r)$, we must have $d(p, \gamma(t)) \geq \frac{r}{2}$, and so
\[ (2 \theta + 1)^2
\geq \min\bigg\{\frac{1}{16}, \frac{1}{4} \cdot \frac{r-1}{16 R^2}\bigg\} \cdot \Big(\frac{r}{2}\Big)^2 \geq \frac{1}{2^9} \frac{r^3}{R^2}. \]
where we have used inequalities~\eqref{eq_case1} and~\eqref{eq_case2}, and the assumption $r \geq 2$.
\end{proof}

\begin{remark}\label{rem:extensionballlike}
As mentioned in~\cref{sec:introduction}, \cref{thm:ball_like} extends to (a) hyperbolic space of any dimension and (b) positive definite matrices with the affine-invariant metric, as both contain the hyperbolic plane as totally geodesic submanifolds.

In fact, following along the lines of~\citep[App I]{bumpfctpaper}, it is easy to extend \cref{thm:ball_like} to any Hadamard manifold whose sectional curvatures are contained in the interval $[K_{\text{lo}}, K_{\text{hi}}]$ with $K_{\text{lo}} \leq K_{\text{hi}} < 0$.
\end{remark}

Appealing to symmetry, we easily get a lower bound for balls.

\begin{proof}[Proof of \cref{cor:ball} (Balls)]
    Let $\gamma$ be the geodesic passing through the ball's center $c$ which is orthogonal to $\nabla F(c)$ (if $\nabla F(c) = 0$, choose any geodesic).
    Let $\tilde F$ be the reflection of $F$ across~$\gamma$.
    Then $\tilde F$ is also a $\theta$-barrier for $D = B(c, R)$, and $F + \tilde F$ is a $2\theta$-barrier for $D$ by~\ref{P4}, with analytic center $c$ (because $(\nabla(F+\tilde F)(c) = 0$).
    Apply \cref{thm:ball_like} to $F + \tilde F$ with $r = R$.
\end{proof}

%-----------------------------------------------------------------------------
\subsection{Horoballs and horospherically convex sets}\label{sec:horoball}
%-----------------------------------------------------------------------------
Let us use \cref{thm:ball_like} to prove a lower bound on $\theta$ for horospherically convex sets on $\HH^2$.
We start with the proof of \cref{cor:horoball} (for horoballs) to illustrate the main ideas.

\begin{proof}[Proof of \cref{cor:horoball} (Horoballs)]
    For contradiction, suppose there is a $\theta$-barrier for some horoball in $\HH^2$.
    Then, by symmetry, every horoball has a $\theta$-barrier.
    Fix $\analyticcenter \in \HH^2$ and $v \in T_{\analyticcenter} \HH^2$ with $\norm{v} = 1$.
    Take $r > 0$ such that $8\theta < \frac{1}{2^6}\cdot \sqrt{r} \cdot \frac{r}{r+1} - \frac{1}{2}$, and consider two geodesic rays $\gamma_\pm \colon [0, \infty) \to \HH^2$ given by
    \[ \gamma_\pm(t) = \exp_{\analyticcenter}(\pm (t - r)). \]
    Let $D_\pm$ be the two horoballs associated to these rays.
    In the Poincar\'e disk model (placing $\analyticcenter$ at the origin), the horoballs are two Euclidean balls tangent to the unit circle, with non-empty intersection: see \cref{fig:test2}.
    Let $D = D_+ \cap D_-$.
    From \cref{fig:test2}, it is immediate that $D$ contains a ball of radius $r$ centered at $\analyticcenter$.
    From that figure and a simple calculation in the Poincar\'e disk model,\footnote{Specifically: the \emph{Euclidean} radii of the inner and outer balls in the disk model are $\tanh(r/2)$ and $\tanh(R/2)$, respectively.  Euclidean geometry then tells us that $\tanh(R/2) = \sqrt{\tanh(r/2)}$.  Solving for $R$ yields $R = \arcosh(e^r)$.} we also find that $D$ is contained in a ball of radius $R = \arcosh(e^r) \leq r+1$ centered at $\analyticcenter$.

    By~\ref{P4}, $D$ has a $2\theta$-barrier~$F$.
    Note that $D$ has reflection symmetries along two orthogonal geodesics (the horizontal and vertical axes in \cref{fig:test2}).
    Applying these symmetries to $F$ generates four $2\theta$-barriers for $D$, whose sum is an $8\theta$-barrier for~$D$ (again by~\ref{P4}) with analytic center $\analyticcenter$.
    Thus \cref{thm:ball_like} yields $8\theta \geq \frac{1}{2^6}\cdot \sqrt{r} \cdot \frac{r}{r+1} - \frac{1}{2}$, which is the desired contradiction.
\end{proof}

Now we move on to general horospherically convex sets.
The idea is to bound the set by a supporting horoball, and then use reflections as in the proof of \cref{cor:horoball}.
\begin{proof}[Proof of \cref{cor:horoconvexset}]
    Let $D$ be a horospherically convex set and suppose that there exists $\analyticcenter \in D$ and $r > 0$ with $B(\analyticcenter, r) \subseteq D$.
    Let $\tilde{r} \geq r$ be the radius of the largest ball centered at $p_*$ contained in $D$ ($\tilde{r}$ is finite since $D \neq \HH^2$).
    The closed ball $\bar{B}(\analyticcenter, \tilde{r})$ must intersect the boundary of $D$ at some point $q$.
    Since $D$ is horospherically convex, there is a supporting horoball $D_+ \supseteq D$ whose boundary contains $q$.

    Let $\gamma$ be the geodesic line passing through $\analyticcenter$ and $q$, and let $\gamma_\perp$ be the geodesic line passing through $\analyticcenter$ and orthogonal to $\gamma$.
    Reflect $D_+$ across $\gamma_\perp$ to obtain a new horoball $D_-$.
    Using the computation from the proof of \cref{cor:horoball}, $D_+ \cap D_-$ is contained in $B(\analyticcenter, R)$ where $R = \arcosh(e^{\tilde{r}}) \leq \tilde{r}+1 \leq 2 \tilde{r}$.

    Consider the group of four symmetries obtained from reflections across $\gamma$ and $\gamma_\perp$.
    Apply these symmetries to $D$ to obtain $D_0 = D, D_1, D_2, D_3$.
    Then $D_0 \cap D_1 \cap D_2 \cap D_3$ is contained in $D_+ \cap D_- \subseteq B(p, R)$, and also contains $B(p, \tilde{r})$.
    Moreover, if $D$ has a $\theta$-barrier, then $D_0 \cap D_1 \cap D_2 \cap D_3$ has a $4\theta$-barrier whose analytic center is $\analyticcenter$.
    Applying \cref{thm:ball_like}, we deduce
    \[
    4 \theta \geq \frac{1}{2^6} \cdot \sqrt{\tilde{r}} \cdot \frac{\tilde{r}}{R} - \frac{1}{2} \geq \frac{1}{2^7} \sqrt{\tilde{r}} - \frac{1}{2} \geq \frac{1}{2^7} \sqrt{r} - \frac{1}{2}. \qedhere
    \]
\end{proof}

%=============================================================================
\section{Perspectives}\label{Sec:futuredirections}
%=============================================================================
Our results show that negative curvature imposes general obstructions on the existence of good barriers for IPMs.
In particular, IPMs based on barriers for the ball are insufficient to solve challenging g-convex optimization problems such as tensor scaling in polynomial time, when the diameter bounds are exponential in the input size.
This opens several potential directions for designing better IPMs:
\begin{itemize}
\item
\textbf{Can one construct self-concordant barriers with low barrier parameter for sublevel sets of Kempf--Ness functions that are bounded below?%
\footnote{This is the setting relevant for scaling problems~\citep{burgissernoncommutativeoptimization}, and excludes Busemann functions (which are not bounded below), whose sublevel sets are horoballs and hence do not admit good barriers (\cref{cor:horoball}).}}
It is known that the sublevel sets of Kempf--Ness functions can have diameter exponential in the input size~\citep{franksreichenback2021}. However, it is unknown whether these sets contain large inner balls---a necessary condition for applying~\cref{thm:ball_like}.
A natural starting point for addressing this question is to explicitly construct a self-concordant barrier for \emph{commutative} scaling problems whose barrier parameter does not depend on the number of weights. While the existence of such barriers is known (e.g., the entropic barrier), no explicit construction is currently available.

\item \label{direction2} \textbf{An alternative---though more speculative---approach is to develop a theory of IPMs based on horospherical convexity, rather than geodesic convexity.}
Horospherically convex optimization admits query complexity bounds that are independent of curvature~\citep{lewis2024horoballs,goodwin2024subgradient,criscitiello2025horosphericallyconvexoptimizationhadamard}.
In particular, horospherically convex functions can exhibit constant condition numbers (independent of the domain's radius) if the condition number is defined in a horospherical---not geodesic---sense; see~\citep[Sec 3.3]{criscitiello2025horosphericallyconvexoptimizationhadamard} and contrast with \citep[Thm.~6]{hamilton2021nogo}.
This motivates finding horospherical definitions of barriers that can be well-behaved even for large horospherically convex domains such as balls.

To apply such a hypothetical theory to scaling problems, one would need to establish that Kempf--Ness functions can be reformulated in the horospherical setting.
There is some precedent: certain nontrivial scaling problems---such as Horn's problem---already fit naturally into the horospherical convexity framework~\citep[Sec.~4.1]{criscitiello2025horosphericallyconvexoptimizationhadamard}.
\end{itemize}

%=============================================================================
\section*{Acknowledgments}
%=============================================================================
MW thanks the Simons Institute for the Theory of Computing at Berkeley and Q-FARM at Stanford University for hospitality.

%=============================================================================
% \section*{Statements and Declarations}
\section*{Funding}
%=============================================================================
%\textbf{Funding:} 
CC was supported by the Swiss State Secretariat for Education, Research and Innovation (SERI) under contract number MB22.00027.
HN acknowledges financial support from the European Research Council (ERC Grant Agreement
No. 818761 and QInteract, Grant No. 101078107) and VILLUM FONDEN via the QMATH Centre of Excellence (Grant No. 10059).
MW acknowledges support by the European Research Council (ERC Grant SYMOPTIC, 101040907), the German Federal Ministry of Research, Technology and Space (QuSol, 13N17173), and the Deutsche Forschungsgemeinschaft (DFG, German Research Foundation, 556164098).

% \medskip

% \noindent \textbf{Conflicts of interest:} The authors have no relevant financial or non-financial interests to disclose.

\appendix
%=============================================================================
\section{Regular polygons}\label{sec:regular-n-gons}
%=============================================================================
Let~$n \geq 3$. In this section, we prove the following.
\begin{theorem}
    \label{thm:regular-n-gon}
    Let $D$ be the interior of a regular $n$-gon in the hyperbolic plane with side lengths~$L > 0$. If~$F$ is a~$\theta$-barrier for~$D$, then~$\theta \geq \frac{L}{32 \log(n)} - \frac{1}{8}$.
\end{theorem}

While the proof of~\cref{thm:triangle} extends in a straightforward way to this setting, that argument yields a weaker bound in which~$\theta$ scales linearly in~$1/n$, rather than~$1/\log n$ as in Theorem~\ref{thm:regular-n-gon}.

We distinguish two cases corresponding to the parity of~$n$. In both cases, we symmetrize the barrier~$F$ to control the position of the analytic center.

\begin{itemize}
\item \textbf{Odd case:} We use a single reflection to control the location of the analytic center.
We then identify a geodesic passing through the analytic center such that one side of the geodesic remains in the domain for a ``long'' time, while the opposite side remains in the domain for only a ``short'' time.

\item \textbf{Even case:} To control the analytic center, we apply either four rotations when $4 \mid n$ (i.e., $4$ divides $n$), or two orthogonal reflections when $n \geq 6$.
In this setting, we construct two geodesics through the analytic center whose Hessian evaluations are comparable, but where one stays in the domain for a long time and the other for only a short time.
\end{itemize}
Theorem~\ref{thm:regular-n-gon} follows immediately from Corollary~\ref{cor:odd-n-gon-lower-bound} (the odd case), and Propositions~\ref{prop:4dividesn} and~\ref{prop:bound-ngeq6-even} (the even case).

%-----------------------------------------------------------------------------
\subsection{Trigonometry of regular \texorpdfstring{$n$}{n}-gons}
%-----------------------------------------------------------------------------
We first establish a simple geometric fact on how close the sides of a regular~$n$-gon come to its centroid, analogous to the proof of~\cref{thm:triangle}.
\begin{lemma}
    \label{lem:regular-n-gon-distance-to-center}
    Let~$p_1, \dotsc, p_n$ be the vertices of a regular~$n$-gon with side lengths~$L > 0$.
    Let~$c$ be the center of the regular~$n$-gon.
    Then, for any~$i \in [n]$,
    the distance from~$c$ to~$p_i$ satisfies
    \[
        R_n \coloneqq d(c, p_i)
        = \arsinh \frac{\sinh(L/2)}{\sin(\pi/n)}
        % = \arcosh \sqrt{\frac{\cosh(L) - \cos(\frac{2\pi}{n})}{1 - \cos(\frac{2\pi}{n})}},
    \]
    and the distance between~$c$ and the line segment from~$p_i$ to~$p_{i+1}$ satisfies
    \[
        r_n \coloneqq d(c, p_i p_{i+1})
        = \arsinh \frac{\tanh(L/2)}{\tan(\pi/n)}
        % = \arcosh \sqrt{\frac{\cosh(L) - \cos(\frac{2\pi}{n})}{(1 - \cos(\frac{2\pi}{n})) \cosh(L/2)^2}}.
    \]
\end{lemma}
To prove the lemma, we use the following fact from hyperbolic trigonometry:
\begin{lemma}[{\cite[Ex.~3.5.4]{ratcliffeFoundationsHyperbolicManifolds2019}}]
  \label{lem:hyperbolic-trigonometry-right-angles}
  Consider a geodesic triangle in~$\HH^2$ with vertices~$q_1, q_2, q_3$, with side lengths~$d(q_1,q_2) = a$, $d(q_2,q_3) = b$ and~$d(q_1,q_3) = c$.
  Assume~$\angle_{q_2}(q_1, q_3) = \pi/2$, and write~$\alpha = \angle_{q_3}(q_1,q_2)$.
  Then
  \[
    \cos \alpha = \frac{\tanh(b)}{\tanh(c)}, \quad \sin \alpha = \frac{\sinh(a)}{\sinh(c)}, \quad \tan \alpha = \frac{\tanh(a)}{\sinh(b)}.
  \]
\end{lemma}
\begin{proof}[Proof of~\cref{lem:regular-n-gon-distance-to-center}.]
    Let~$q_i$ be the midpoint of~$p_i$ and~$p_{i+1}$.
    Then~$d(c, q_i) = d(c, p_i p_{i+1}) = r_n$, $d(q_i, p_i) = L/2$ and~$d(c, p_i) = R_n$.
    Applying~\cref{lem:hyperbolic-trigonometry-right-angles} to the hyperbolic triangle~$c q_i p_i$, and using that~$\angle_c(q_i, p_i) = \pi/n$, we obtain the desired result immediately from the identities~$\sin(\pi/n) = \sinh(L/2) / \sinh(R_n)$ and~$\tan(\pi/n) = \tanh(L/2) \sinh(r_n)$.
    % The isosceles triangle~$p_i c p_{i+1}$ has angle~$2\pi/n$ at~$c$ and $p_i p_{i+1}$ has side length~$L$, so the hyperbolic law of cosines yields
    % \begin{align*}
    %     \cosh(L) & = \cosh(d(c,p_i)) \cosh(d(c,p_{i+1})) - \sinh(d(c,p_i)) \sinh(d(c,p_{i+1})) \cos(\frac{2\pi}{n}) \\
    %     & = \cosh(d(c,p_i))^2 - \sinh(d(c,p_i))^2 \cos(\frac{2\pi}{n}) \\
    %     & = \cosh(d(c,p_i))^2 - (1 - \cosh(d(c,p_i))^2) \cos(\frac{2\pi}{n}),
    % \end{align*}
    % hence
    % \[
    %     (1 - \cos(\frac{2\pi}{n})) \cosh(R_n)^2 = \cosh(L) - \cos(\frac{2\pi}{n}),
    % \]
    % which yields the desired expression for~$R_n$.
    % Next, let $q_i$ be the midpoint of~$p_i p_{i+1}$, then~$r_n = d(c, p_i p_{i+1}) = d(c, q_i)$ and
    % \[
    %     \cosh(R_n) = \cosh(r_n) \cosh(L/2),
    % \]
    % so
    % \[
    %     \cosh(r_n)^2 = \frac{\cosh(L) - \cos(\frac{2\pi}{n})}{(1 - \cos(\frac{2\pi}{n})) \cosh(L/2)^2}.
    % \]
\end{proof}

The following bounds for~$R_n$ and~$r_n$ follow immediately from Lemma~\ref{lem:hyperbolic-trigonometry-right-angles}; we omit the proof.
\begin{lemma}\label{lem:boundsonRr}
    Under the setting of Lemma~\ref{lem:regular-n-gon-distance-to-center}, we have
    \[
    R_n \geq \frac{L}{2}, \quad r_n \leq \arsinh(1/\tan(\pi/n)) \leq \log(n).
    \]
\end{lemma}

%-----------------------------------------------------------------------------
\subsection{The odd case}
%-----------------------------------------------------------------------------
Let~$n \geq 3$ be odd, and let~$D$ be a regular~$n$-gon in~$\HH^2$ with side length~$L$.
Let $c$ be the centroid of the regular $n$-gon with side lengths~$L$ and vertices~$p_1, \dotsc, p_n$.
Let~$q_i$ be the midpoint of~$p_i$ and~$p_{i+1}$.
Let~$\hat{p}_i$ be the point opposite~$q_i$.
\begin{proposition}\label{helperprop}
    There is an~$i \in [n]$ such that~$F^i(x) = F(x) + F(O_i x)$ has analytic center on the geodesic segment from~$c$ to~$q_i$, where~$O_i$ is the reflection of~$\HH^2$ which preserves the geodesic passing through~$q_i$ and~$\hat{p}_i$.
\end{proposition}
\begin{proof}
    Choose~$i \in [n]$ satisfying~$\langle\nabla F(c), \gamma_i'(0)\rangle \leq 0$, where $\gamma_i\colon [0,1]\to D$ is the geodesic from $c$ to $q_i$. Along~$\gamma_i$ we know~$F(\gamma_i(t)) \to \infty$ as~$t \to 1$, so it must achieve a minimum at some~$t_*$ on $[0,1)$. We then have~$\langle \nabla F(\gamma_i(t_*)), \gamma_i'(t_*) \rangle = 0$.
    Since the reflection~$O_i$ preserves~$\gamma_i$ and applies a reflection on the tangent space~$T_{\gamma_i(t)} \HH^2$, we deduce that~$\nabla (F \circ O_i)(\gamma_i(t)) = -\nabla F(\gamma_i(t))$, and therefore~$\nabla F^i(\gamma_i(t)) = 0$.
\end{proof}

\begin{corollary}
    \label{cor:odd-n-gon-lower-bound}
    Let~$F$ be a~$\theta$-barrier for~$D$.
    Then~$4 \theta + 1 \geq R_n/r_n \geq \frac{1}{2} L / \log(n)$.
\end{corollary}
\begin{proof}
    Let~$\analyticcenter$ be the analytic center of~$F^i$, where $i \in [n]$ is as in Proposition~\ref{helperprop}.
    The distance $d(\analyticcenter, q_i)$ is at most~$r_n$, and the distance~$d(\analyticcenter, \hat{p}_i)$ is at least~$R_n$.
    Now let~$u \in T_{\analyticcenter} \HH^2$ be the unit-speed direction pointing to~$\hat{p}_i$.
    Then~$\log_{\analyticcenter}(\hat{p}_i) = t_+ u$ and~$\log_{\analyticcenter}(q_i) = - t_- u$ for some~$t_+ \geq R_n$ and~$0 \leq t_- \leq r_n$.
    Then
    \[
        t_+^2 (\nabla^2 F^i)_{\analyticcenter}[u,u] \leq (4 \theta + 1)^2
    \]
    by~\ref{P3}, and
    \[
        t_-^2 (\nabla^2 F^i)_{\analyticcenter}[-u,-u] \geq 1
    \]
    by~\ref{P2}.
    Combining these estimates yields
    \[
    (4 \theta + 1)^2 \geq t_+^2 (\nabla^2 F^i)_{\analyticcenter}[u,u] = \frac{t_+^2}{t_-^2} \, t_-^2 (\nabla^2 F^i)_{\analyticcenter}[-u,-u] \geq \frac{t_+^2}{t_-^2} \geq \frac{R_n^2}{r_n^2}.
    \]
    Lastly, Lemma~\ref{lem:boundsonRr} gives $R_n/r_n \geq \frac{1}{2} L / \log(n)$.
\end{proof}

%-----------------------------------------------------------------------------
\subsection{The even case}
%-----------------------------------------------------------------------------
For even~$n$ we also exploit symmetries of the domain to control the location of the analytic center.
First, we deal with the case~$4 \mid n$.

\begin{proposition}\label{prop:4dividesn}
    Let~$n = 4k$, with $k$ a positive integer.
    Let~$F$ be a~$\theta$-barrier for~$D$.
    Then
    \[ 8 \theta + 1 \geq \frac{R_n}{r_n} \geq \frac{L}{2 \log(n)}. \]
\end{proposition}
\begin{proof}
    Let~$\mathbb{Z}/4\mathbb{Z}$ act on~$\mathbb{H}^2$ by rotation through an angle of~$\pi/2$ about the center~$c$ of~$D$.
    This action restricts to an action on~$D$, and its only fixed point is~$c$.
    Letting~$\tilde F(x) = \sum_{g \in \ZZ/4\ZZ} F(g \cdot x)$,~$\tilde F$ is a~$4\theta$-barrier for~$D$ with analytic center~$c$ by \ref{P4}.

    The action of~$\mathbb{Z}/4\mathbb{Z}$ on~$\mathbb{H}^2$ induces an action on~$T_c \mathbb{H}^2 \cong \mathbb{R}^2$ (rotation by~$\pi/2$), under which $(\nabla^2 \tilde{F})_c$ is invariant.
    Dualizing~$(\nabla^2 \tilde F)_c$ to the Hessian operator~$\Hess_c \tilde F$, represented by a real~$2 \times 2$ matrix in an orthonormal basis, we obtain the equation
    \[
        \begin{bmatrix}
            0 & -1 \\
            1 & 0
        \end{bmatrix}^\top
        \Hess_c \tilde F
        \begin{bmatrix}
            0 & -1 \\
            1 & 0
        \end{bmatrix}
        =
        \Hess_c \tilde F.
    \]
    This implies that $\Hess_c \tilde F = \lambda I$ for some~$\lambda \in \RR$.
    Moreover, $\lambda \ge 0$, since otherwise $\tilde{F}$ would fail to be convex; and in fact $\lambda > 0$ by~\ref{P3} together with the boundedness of~$D$.

    Let~$u, v \in T_c \HH^2$ be the unit-speed directions towards~$p_1$ and~$q$ respectively, where~$q$ is the midpoint of~$p_1$ and~$p_2$.
    Then
    \[
        \lambda d(c,p_1)^2 = d(c,p_1)^2 (\nabla^2 \tilde F)_c[u, u] = (\nabla^2 \tilde F)_c[\log_c(p_1), \log_c(p_1)] \leq  (8 \theta + 1)^2.
    \]
    by \ref{P3}.
    In the other direction, \ref{P2} implies that
    \[
        \lambda d(c,q)^2 = d(c,q)^2 (\nabla^2 \tilde F)_c[v, v] = (\nabla^2 \tilde F)_c[\log_c(q), \log_c(q)] \geq 1.
    \]
    We conclude that
    \[
        8 \theta + 1 \geq \sqrt{\lambda} d(c,p_1) \geq \frac{d(c,p_1)}{d(c,q)} = \frac{R_n}{r_n}.
    \]
    Lastly, Lemma~\ref{lem:boundsonRr} gives $R_n/r_n \geq \frac{1}{2} L / \log(n)$.
\end{proof}

We now turn to the bound for even~$n \geq 6$.
\begin{figure}
  \begin{center}
    \begin{minipage}{.3\textwidth}
      \input{Fig4_hexagon.tex}
    \end{minipage}
    \begin{minipage}{.35\textwidth}
      \input{Fig4_8gon.tex}
    \end{minipage}
    \begin{minipage}{.3\textwidth}
      \input{Fig4_10gon.tex}
    \end{minipage}
  \end{center}
  \caption{Schematic representations of the points used in the proof of~\cref{prop:bound-ngeq6-even} for $n=6$, $n=8$ and~$n=10$. The point~$c$ is the centroid of the regular~$n$-gon.  The point~$q$ is the geodesic midpoint of~$p_{\lfloor n/4 \rfloor + 1}$ and $p_{\lceil n/4 \rceil + 1}$.
    The point~$q'$ is the geodesic midpoint of~$p_{j-1}$ and~$p_j$, where $j = \lceil n/8 \rceil + 1$ is~$2$, $2$, $3$, respectively.
  The dashed lines indicate the directions~$u \sim \log_c(p_j), v \sim \log_c(q') \in T_c \HH^2$.}
\end{figure}
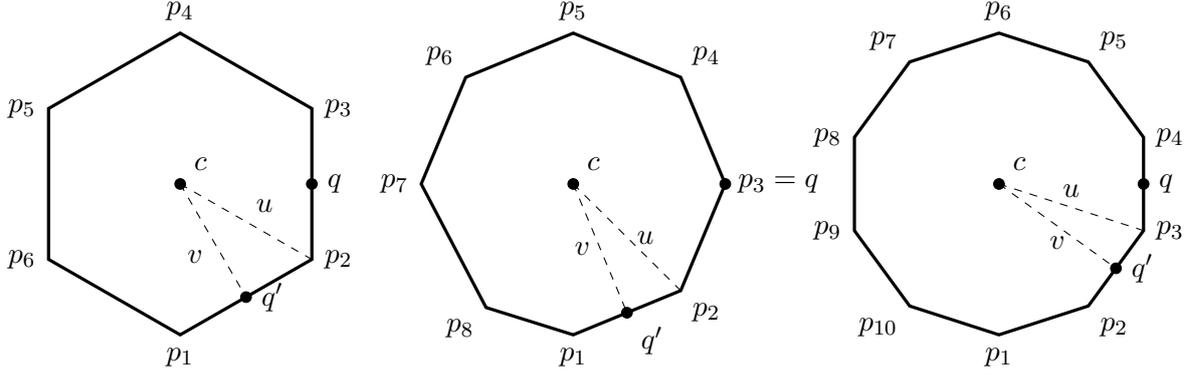
\begin{proposition}
  \label{prop:bound-ngeq6-even}
    Let~$n \geq 6$ be even.
    Let~$F$ be a~$\theta$-barrier for~$D$.
    Then
    \[ 8 \theta + 1 \geq \frac{R_n}{2 r_n} \geq \frac{L}{4 \log(n)}. \]
\end{proposition}
\begin{proof}
    Let~$c$ be the center of~$D$.
    Let~$s_1$ be the reflection of~$\HH^2$ which preserves the geodesic passing through~$c$ and~$p_1$ (i.e., the geodesic through $c$ and $p_1$ is invariant under the reflection).
    Let~$q$ be the midpoint of~$p_{\lfloor n/4 \rfloor + 1}$ and~$p_{\lceil n/4 \rceil + 1}$.
    Let~$s_2$ be the reflection of~$\HH^2$ which preserves the geodesic passing through~$c$ and~$q$.
    The reflections~$s_1$ and $s_2$ are orthogonal (as $\angle_c(p_1, q) = \pi/2$), hence they commute.
    Moreover both $s_1$ and $s_2$ restrict to isometries on $D$.

    Next, define
    \[
        \tilde F(x) = \sum_{b_1, b_2 = 0}^1 F(s_1^{b_1} s_2^{b_2} x)
    \]
    which is a $4\theta$-barrier for $D$ by \ref{P4}, and invariant under both reflections.
    The only common fixed point of the reflections is $c$, hence $c$ is the analytic center of $D$ with respect to $\tilde F$.

    The induced action on $T_c \HH^2$ leaves $(\nabla^2 \tilde F)_c$ invariant.
    Letting $w, w' \in T_c \HH^2$ denote the unit speed direction from~$c$ to~$p_1$ and~$q$, respectively, we see that $(\nabla^2 \tilde F)_c$ is represented by a matrix
    \[
        \begin{bmatrix}
            \lambda & 0 \\
            0 & \lambda'
        \end{bmatrix}
    \]
    in the orthonormal basis~$w, w'$.

    We will now use that~$n \geq 6$. Set $j = \lceil n/8 \rceil + 1$.
    Let~$q'$ be the midpoint of~$p_j$ and~$p_{j-1}$.
    We prove two estimates on the angles appearing in the $n$-gon, and then wrap up the proof.

\paragraph{First estimate:} Note that~$\angle_c(q, p_j) + \frac{\pi}{n} = \angle_c(q, q') \leq \frac{\pi}{2}$.
    As the cosine is decreasing on~$[0,\pi]$, we conclude that~$0 \leq \cos(\angle_c(q, q')) \leq \cos(\angle_c(q, p_j))$.

\paragraph{Second estimate:} Next, we show that~$\cos(\angle_c(p_1,p_j))^2 \geq \frac{1}{4} \cos(\angle_c(p_1, q'))^2$.
    For~$n \leq 10$, we verify manually:
    \begin{center}
        \begin{tabular}{c|c|c|c}
        $n$ & $\angle_c(p_1,p_j)$ & $\angle_c(p_1,q')$ & $\frac{\cos(\angle_c(p_1,p_j))^2}{\cos(\angle_c(p_1,q'))^2}$ \\
        \hline
        $6$ &
        $\pi/3$ & %$\frac{\pi}{3}$ &
        $\pi/6$ & %$\frac{\pi}{6}$ &
        $1/3$ \\%$\frac{1}{3}$ \\
        $8$ &
        $\pi/4$ & %$\frac{\pi}{4}$ &
        $\pi/8$ & %$\frac{\pi}{8}$ &
        $\approx 0.58$ \\
        $10$ &
        $2\pi/5$ & % $\frac{2\pi}{5}$ &
        $3\pi/10$ & % $\frac{3\pi}{10}$ &
        $\approx 0.28$
        \end{tabular}
        \end{center}
    % \begin{align*}
    %     n = 6: && \angle_c(p_1,p_j) = \frac{\pi}{3} && \angle_c(p_1,q') = \frac{\pi}{6} && \frac{\cos(\angle_c(p_1,p_j))^2}{\cos(\angle_c(p_1,q'))^2} = \frac{1}{3} \\
    %     n = 8: && \angle_c(p_1,p_j) = \frac{\pi}{4} && \angle_c(p_1,q') = \frac{\pi}{8} && \frac{\cos(\angle_c(p_1,p_j))^2}{\cos(\angle_c(p_1,q'))^2} \approx 0.58 \\
    %     n = 10: && \angle_c(p_1,p_j) = \frac{2\pi}{5} && \angle_c(p_1,q') =
    %     %\frac{2\pi}{5} - \frac{2\pi}{20} =
    %     \frac{3\pi}{10} && \frac{\cos(\angle_c(p_1,p_j))^2}{\cos(\angle_c(p_1,q'))^2} \approx 0.28.
    % \end{align*}
    For~$n \geq 12$, set $\alpha = \angle_c(p_1,p_j) = \angle_c(p_1,q') + \frac{\pi}{n}$.
    Then
    \[
    \frac{\pi}{4} \leq \alpha \leq \frac{2\pi}{n} \Big(\frac{n}{8} + 1\Big) \leq \frac{\pi}{4} + \frac{2\pi}{n} \leq \frac{\pi}{4} + \frac{\pi}{6}.
    \]
    Now~$\sin(x) \leq C \cos(x)$ for~$x \leq \frac{\pi}{4} + \frac{\pi}{6} = \frac{5\pi}{12}$ with~$C = \tan(\frac{5\pi}{12}) \approx 3.73$.
    Therefore,
    \begin{align*}
    \cos(\alpha - \pi / n)^2 &= (\cos(\alpha) \cos(\pi/n) + \sin(\alpha) \sin(\pi / n))^2 \\
    &\leq (\cos(\alpha) \cos(\pi/n) + C \cos(\alpha) \sin(\pi / n))^2 \\
    & = \cos(\alpha)^2 (\cos(\pi/n) + C \sin(\pi / n))^2
    \leq 4 \cos(\alpha)^2.
    \end{align*}
    The last inequality may be shown by observing that~$(\cos(\pi/n) + C \sin(\pi/n))^2$ is decreasing as a function of~$n$, and equals $C = \tan(\frac{5\pi}{12}) \leq 3.74$ for~$n = 12$.
    Therefore we have verified for all~$n\geq 6$ that~$\cos(\angle_c(p_1,p_j))^2 \geq \frac14 \cos(\angle_c(p_1,q'))^2$.
    % then we also have
    % \[
    %     \frac{\pi}{4} - \frac{\pi}{n} = \frac{2\pi}{8} - \frac{2\pi}{2n} \leq \angle_c(p_1, q') = \frac{2\pi}{n} \left\lceil \frac{n}{8} \right\rceil + \frac{\pi}{n} \leq \frac{2\pi}{8} + \frac{2\pi}{2n} = \frac{\pi}{4} - \frac{\pi}{n}.
    % \]
    % This implies that
    % \[
    %     3 \geq \frac{\cos(\pi/4)}{\cos(\pi/4 + \pi/n)} \geq
    %     \frac{\cos(\angle_c(p_1, p_j))}{\cos(\angle_c(p_1, q'))} \geq \frac{\cos(\pi/4 + 2\pi/n)}{\cos(\pi/4 - \pi/n)}
    % \]
    % Moreover observe that~$\angle_c(p_1, p_j) = \frac{\pi}{4} - \angle_c(q, p_j)$ and similar for~$q'$.

\paragraph{Wrapping up:}
Let~$\lambda,\lambda'$ be the eigenvalues of~$(\nabla^2 \tilde F)_c$, and~$w,w'$ the associated eigenvectors.
    Then, letting~$u, v$ be the unit-speed directions from $c$ to $p_j$ and~$q'$ respectively, we see that
    \begin{align*}
        (\nabla^2 \tilde F)_c[u,u] & = \lambda_1 \langle u, w_1 \rangle^2 + \lambda_2 \langle u, w_2 \rangle^2 \\
        & = \lambda_1 \cos(\angle_c(p_1, p_j))^2 + \lambda_2 \cos(\angle_c(q, p_j))^2 \\
        & \geq \frac{\lambda_1}{4} \cos(\angle_c(p_1, q'))^2 + \lambda_2 \cos(\angle_c(q, q'))^2 \\
        & \geq \frac14 \left(\lambda_1 \cos(\angle_c(p_1, q'))^2 + \lambda_2 \cos(\angle_c(q, q'))^2\right)
        = \frac14 (\nabla^2 \tilde F)_c[v,v].
    \end{align*}
    But~$R_n^2 (\nabla^2 \tilde F)_c[u,u] = (\nabla^2 \tilde F)_c[\log_c(p_j), \log_c(p_j)] \leq (8 \theta + 1)^2$ and
    \[
    r_n^2 (\nabla^2 \tilde F)_c[v,v] = (\nabla^2 \tilde F)_c[\log_c(q'), \log_c(q')] \geq 1,
    \]
    by \ref{P3} and \ref{P2} respectively.
    Therefore
    \begin{align*}
        (8 \theta + 1)^2 \geq R_n^2 (\nabla^2 \tilde F)_c[u,u]
        = \frac{R_n^2}{r_n^2} \cdot r_n^2 (\nabla^2 \tilde F)_c[u,u]
        \geq \frac{R_n^2}{4 r_n^2} \cdot r_n^2 (\nabla^2 \tilde F)_c[v,v]
        \geq \frac{R_n^2}{4 r_n^2}.
    \end{align*}
    Lastly, Lemma~\ref{lem:boundsonRr} gives $R_n/(2 r_n) \geq \frac{1}{4} L / \log(n)$.
    \end{proof}

\bibliography{references}

\end{document}

%% file: Fig4_hexagon.tex
\begin{tikzpicture}[scale=2]

% Define hexagon vertices
\coordinate (P2) at (-30:1);
\coordinate (P3) at (30:1);
\coordinate (P4) at (90:1);
\coordinate (P5) at (150:1);
\coordinate (P6) at (210:1);
\coordinate (P1) at (270:1);

% Draw hexagon
\draw[very thick] (P1) -- (P2) -- (P3) -- (P4) -- (P5) -- (P6) -- cycle;

% Label vertices
\node[below=1pt] at (P1) {\( p_1 \)};
\node[right=1pt] at (P2) {\( p_2 \)};
\node[right=1pt] at (P3) {\( p_3 \)};
\node[above=1pt] at (P4) {\( p_4 \)};
\node[left=1pt] at (P5) {\( p_5 \)};
\node[left=1pt] at (P6) {\( p_6 \)};

% Define and draw points q and q'
\path let \p1 = (P2), \p2 = (P3) in
  coordinate (Q) at ($(\p1)!0.5!(\p2)$);
\path let \p1 = (P1), \p2 = (P2) in
  coordinate (Q') at ($(\p1)!0.5!(\p2)$);

\filldraw[black] (Q) circle (1pt);
\filldraw[black] (Q') circle (1pt);
\filldraw[black] (0,0) circle (1pt);

% Label q and q'
\node[right=2pt] at (Q) {\( q \)};
\node[right=2pt] at (Q') {\( q' \)};
\node[above right=2pt] at (0,0) {\( c \)};

\coordinate (PJ) at (P2);
\draw[dashed] (0,0) -- (Q');
\draw[dashed] (0,0) -- (PJ);
\node[above right] at ($(0,0)!0.5!(PJ)$) {$u$};
\node[below left] at ($(0,0)!0.5!(Q')$) {$v$};

\end{tikzpicture}

%% file: Fig4_8gon.tex
\begin{tikzpicture}[scale=2]

% Define hexagon vertices
\coordinate (P1) at (270:1);
\coordinate (P2) at (315:1);
\coordinate (P3) at (0:1);
\coordinate (P4) at (45:1);
\coordinate (P5) at (90:1);
\coordinate (P6) at (135:1);
\coordinate (P7) at (180:1);
\coordinate (P8) at (235:1);

% Draw hexagon
\draw[very thick] (P1) -- (P2) -- (P3) -- (P4) -- (P5) -- (P6) -- (P7) -- (P8) -- cycle;

% Label vertices
\node[below=1pt] at (P1) {\( p_1 \)};
\node[below right=1pt] at (P2) {\( p_2 \)};
\node[right=1pt] at (P3) {\( p_3 = q \)};
\node[above right=1pt] at (P4) {\( p_4 \)};
\node[above=1pt] at (P5) {\( p_5 \)};
\node[above left=1pt] at (P6) {\( p_6 \)};
\node[left=1pt] at (P7) {\( p_7 \)};
\node[below left=1pt] at (P8) {\( p_8 \)};

% Define and draw points q and q'
\path let \p1 = (P2), \p2 = (P3) in
  coordinate (Q) at (P3);
\path let \p1 = (P1), \p2 = (P2) in
  coordinate (Q') at ($(\p1)!0.5!(\p2)$);
\coordinate (PJ) at (P2);

\filldraw[black] (Q) circle (1pt);
\filldraw[black] (Q') circle (1pt);
\filldraw[black] (0,0) circle (1pt);

% Label q and q'
% \node[right=2pt] at (Q) {\( q \)};
\node[below right=2pt] at (Q') {\( q' \)};
\node[above right=2pt] at (0,0) {\( c \)};

\draw[dashed] (0,0) -- (Q');
\draw[dashed] (0,0) -- (P2);
% \draw[dashed] (0,0) -- node {$v$} -- (P2);
\node[right] at ($(0,0)!0.5!(PJ)$) {$u$};
\node[left] at ($(0,0)!0.5!(Q')$) {$v$};

\end{tikzpicture}

%% file: Fig4_10gon.tex
\begin{tikzpicture}[scale=2]

% Define hexagon vertices
\coordinate (P1) at (270:1);
\coordinate (P2) at (306:1);
\coordinate (P3) at (342:1);
\coordinate (P4) at (18:1);
\coordinate (P5) at (54:1);
\coordinate (P6) at (90:1);
\coordinate (P7) at (126:1);
\coordinate (P8) at (162:1);
\coordinate (P9) at (198:1);
\coordinate (P10) at (234:1);

% Draw hexagon
\draw[very thick] (P1) -- (P2) -- (P3) -- (P4) -- (P5) -- (P6) -- (P7) -- (P8) -- (P9) -- (P10) -- cycle;

% Label vertices
\node[below=1pt] at (P1) {\( p_1 \)};
\node[below right=1pt] at (P2) {\( p_2 \)};
\node[right=1pt] at (P3) {\( p_3 \)};
\node[right=1pt] at (P4) {\( p_4 \)};
\node[above right=1pt] at (P5) {\( p_5 \)};
\node[above=1pt] at (P6) {\( p_6 \)};
\node[above left=1pt] at (P7) {\( p_7 \)};
\node[left=1pt] at (P8) {\( p_8 \)};
\node[left=1pt] at (P9) {\( p_9 \)};
\node[below left=1pt] at (P10) {\( p_{10} \)};

% Define and draw points q and q'
\path let \p1 = (P3), \p2 = (P4) in
  coordinate (Q) at ($(\p1)!0.5!(\p2)$);
\path let \p1 = (P2), \p2 = (P3) in
  coordinate (Q') at ($(\p1)!0.5!(\p2)$);

\coordinate (PJ) at (P3);

\filldraw[black] (Q) circle (1pt);
\filldraw[black] (Q') circle (1pt);
\filldraw[black] (0,0) circle (1pt);

% Label q and q'
\node[right=2pt] at (Q) {\( q \)};
\node[right=2pt] at (Q') {\( q' \)};
\node[above right=2pt] at (0,0) {\( c \)};

\draw[dashed] (0,0) -- (Q');
\draw[dashed] (0,0) -- (PJ);
% \draw[dashed] (0,0) -- node {$v$} -- (P2);
\node[above] at ($(0,0)!0.5!(PJ)$) {$u$};
\node[below] at ($(0,0)!0.5!(Q')$) {$v$};

\end{tikzpicture}